\documentclass[english, a4paper, 12pt]{article}

\usepackage{babel}

\usepackage[T1]{fontenc}

\usepackage[utf8]{inputenc}

\usepackage[margin=2.5cm]{geometry}

\usepackage{titlesec}
\titleformat{\chapter}{\normalfont\huge}{\thechapter.}{20pt}{\huge\it}

\usepackage[yyyymmdd]{datetime} 

\usepackage{enumerate} 
\usepackage{enumitem} 

\usepackage{csquotes} 
\usepackage[style=alphabetic]{biblatex} 
\bibliography{bibliography} 

\usepackage[hidelinks]{hyperref}

\usepackage{color}
\usepackage{xcolor}

\usepackage{graphics}
\usepackage{graphicx} 
\usepackage{float} 
\usepackage{caption} 
\usepackage{subcaption} 
\usepackage{adjustbox} 

\usepackage{tikz}
\usetikzlibrary{arrows.meta} 
\usepackage{pgf}
\usepackage{pgfplots}
\pgfplotsset{compat=1.18}

\usepackage{pifont}

\usepackage{siunitx}
\DeclareSIUnit\clight{\text{$c$}} 
\DeclareSIUnit\byte{B}

\usepackage[notrig]{physics} 
\usepackage{braket} 
\usepackage{slashed} 
\usepackage{tensor} 

\usepackage{listings}
\definecolor{codegreen}{rgb}{0,0.6,0}
\definecolor{codegray}{rgb}{0.5,0.5,0.5}
\definecolor{codepurple}{rgb}{0.58,0,0.82}
\definecolor{backcolour}{rgb}{0.95,0.95,0.92}
\lstdefinestyle{mystyle}{
    backgroundcolor=\color{backcolour},   
    commentstyle=\color{codegreen},
    keywordstyle=\color{magenta},
    numberstyle=\tiny\color{codegray},
    stringstyle=\color{codepurple},
    basicstyle=\ttfamily\footnotesize,
    breakatwhitespace=false,         
    breaklines=true,                 
    captionpos=b,                    
    keepspaces=true,                 
    numbers=left,                    
    numbersep=5pt,                  
    showspaces=false,                
    showstringspaces=false,
    showtabs=false,
    tabsize=4
}
\usepackage{amsmath}
\usepackage{amsfonts}
\usepackage{amsthm}
\usepackage{amssymb}
\usepackage{mathrsfs} 

\usepackage{mleftright} 
\mleftright

\usepackage{cancel}

\usepackage{polynom}

\usepackage{tikz-cd} 
\tikzcdset{every label/.append style = {font = \normalsize}}

\usepackage{interval}
\intervalconfig{soft open fences}

\usepackage{aligned-overset} 

\let\mathbar\overline

\let\mathhat\widehat

\renewcommand{\exp}[1]{\mathrm{e}^{#1}}

\newcommand*\reals{\mathbb{R}}

\def\from{\colon} 
\newlength\mylen 
\newcommand{\mapsbetween}{\longleftrightarrow\kern - 0.5\mylen\vline height 1.2ex depth -0.0pt\kern0.5\mylen}
\newcommand{\suchthat}{\qq{s.th.}}

\newcommand*{\transpose}[1]{#1^{\mathsf{T}}}

\newcommand*{\inverse}[1]{#1^{-1}}

\newcommand*\actson{\curvearrowright}

\usepackage[noabbrev, nameinlink]{cleveref}
\crefname{equation}{}{}
\iflanguage{swedish}{ 
	
}{}
\iflanguage{english}{
	
}{}

\iflanguage{swedish}{
    \theoremstyle{plain}
    \newtheorem{theorem}{Sats}
    \newtheorem*{theorem*}{Sats}
    \newtheorem{proposition}{Proposition}
    \newtheorem*{proposition*}{Proposition}
    \newtheorem{corollary}{Följdsats}[theorem]
    \newtheorem*{corollary*}{Följdsats}
    \newtheorem{lemma}{Lemma}
    \newtheorem*{lemma*}{Lemma}
	\newtheorem{conjecture}{Förmodan}
	\newtheorem*{conjecture*}{Förmodan}

    \theoremstyle{definition}
    \newtheorem{definition}{Definition}
    \newtheorem*{definition*}{Definition}
	\newtheorem*{example}{Exempel}

    \theoremstyle{remark}
	\newtheorem*{remark}{Kommentar}
}{}
\iflanguage{english}{
    \theoremstyle{plain}
    \newtheorem{theorem}{Theorem}
    \newtheorem*{theorem*}{Theorem}
    \newtheorem{proposition}[theorem]{Proposition}
    \newtheorem*{proposition*}{Proposition}
    \newtheorem{corollary}{Corollary}[theorem]
    \newtheorem{corollary*}{Corollary}
    \newtheorem{lemma}[theorem]{Lemma}
    \newtheorem*{lemma*}{Lemma}
	
	\newtheorem{conjecture*}{Conjecture}

    \theoremstyle{definition}
    \newtheorem{definition}[theorem]{Definition}
    \newtheorem*{definition*}{Definition}

    \theoremstyle{remark}
	\newtheorem*{remark}{Remark}
}{}

\usepackage{xstring}
\usepackage[]{marginnote}
\newcommand{\todo}[2][]{%
	\IfStrEqCase{#1}{%
        {inline}{{\color{red}#2}}%
	}[%
		\marginnote{\color{red}#2}%
	]
}

\usepackage{csvsimple} 

\numberwithin{equation}{section}

\title{A homogeneous geometry of low-rank tensors}
\author{Simon Jacobsson\footnote{Department of Computer Science, KU Leuven, Celestijnenlaan 200A - box 2402, Leuven, 3000, Belgium (\href{mailto:simon.jacobsson@kuleuven.be}{simon.jacobsson@kuleuven.be}). ORCiD: \href{https://orcid.org/0000-0002-1181-972X}{0000-0002-1181-972X}}}
\date{2025-12-15}

\begin{document}
\maketitle

\begin{abstract}
	We consider sets of fixed CP, multilinear, and TT rank tensors, and derive conditions for when (the smooth parts of) these sets are smooth homogeneous manifolds.
	For CP and TT ranks, the conditions are essentially that the rank is sufficiently low.
	These homogeneous structures are then used to derive Riemannian metrics whose geodesics are both complete and efficient to compute.
\end{abstract}

\section{Introduction}%
\label{sec:Introduction}

Identifying fixed-rank matrices as quotient manifolds has had very profitable applications in manifold optimization and statistics~\cite{Bonnabel10,Journee10,Bonnabel12,Mishra14,Massart20,Zheng25}.
For example, Vandereycken, Absil, and Vandewalle's~\cite{Vandereycken12} identification of fixed-rank symmetric matrices as a \emph{homogeneous manifold} is particularly useful because it induces a Riemannian geometry with complete geodesics.
A natural question is whether such a construction generalizes to tensors.

Consider the group action that multiplies tensors in each mode by an invertible matrix, which can also be seen as a change of basis in each mode.
While there are several ways to define the \emph{rank} of a tensor, for most of them, the rank is invariant under this action.
There are in general two obstacles to using it to identify sets of fixed-rank tensors as homogeneous manifolds:
\begin{enumerate}
	\item 
	Without modification, the action is not transitive since there is no way to go from tensors on the form $a \otimes a \otimes a + b \otimes b \otimes b$ to $a \otimes a \otimes a + a \otimes b \otimes b$.
	\item
	We need to calculate the stabilizer, which relates to the uniqueness of the corresponding rank decomposition.
	However, this is much more involved for tensors than for matrices.
	The degree of uniqueness, or \emph{identifiability}, of tensor decompositions is often studied from an algebraic geometric perspective, and broader results have to assume some symmetry, small size, or low rank.
	See for example \cite{Chiantini14,Blomenhofer24}.
\end{enumerate}

In this paper, we consider three notions of rank: \emph{canonical polyadic} (CP) rank, multilinear rank, and \emph{tensor train} (TT) rank.
The two obstacles are overcome by:
\begin{enumerate}
	\item
	Restricting to a Zariski open subset where the action is transitive.
	\item
	Restricting to ranks where we understand identifiability.
\end{enumerate}
We are thus able to identify three families of fixed-rank tensor manifolds as smooth homogeneous manifolds.
For each of these, there is a so-called \emph{canonical Riemannian metric}.
We show how the geodesics in this metric can be computed efficiently.

\subsection*{Related work}

As already mentioned, the space of matrices with fixed rank is a homogeneous manifold~\cite{Vandereycken12,Absil15,MuntheKaas15}.

In low-rank approximation, the manifold perspective has been used to study several different fixed-rank tensor spaces.
We mention here CP rank~\cite{Breiding18}, TT rank~\cite{Holtz12,Uschmajew20}, hierarchical Tucker rank~\cite{Uschmajew13}, and multilinear rank~\cite{Koch10}.

In \cite{Uschmajew20}, it is mentioned that \enquote{the set of all tensors with canonical rank bounded by $k$ is typically not closed.
Moreover, while the closure of this set is an algebraic variety, its smooth part is in general not equal to the set of tensors of fixed rank $k$ and does not admit an easy explicit description.}.
We contribute an \enquote{easy explicit description} for a subset of those manifolds.

We also mention that fixed-rank tensor manifolds have, to our knowledge, not previously been equipped with a Riemannian metric with known geodesics, other than in the rank $1$ case~\cite{Swijsen22b,Jacobsson24}.

In Riemannian optimization, line searches and gradient descents are often performed along geodesics, or along approximate geodesics, called \emph{retractions}.
As already mentioned, on any homogeneous manifold there is a distinguished Riemannian metric called the canonical metric.
This metric has two big advantages: geodesics on a quotient can be described as geodesics on its numerator; and those geodesics are always complete, meaning they can be extended to arbitrary length.
Similarly, a retraction on the quotient can be induced by a retraction on the numerator.
Riemannian optimization using the canonical metric has been explored, for example, on the Grassmann~\cite{Helmke07,Sato14,Boumal15,Bendokat24}, Stiefel~\cite{Edelman98,Li20,Gao22}, and symplectic Stiefel~\cite{Gao21,Bendokat21} manifolds.

\emph{Lie group integrators} are a class of numerical integrators for ordinary differential equations (ODEs) on manifolds~\cite{Iserles00,Christiansen11,Blanes09}.
They work by replacing the manifold ODE with an appropriate Lie group ODE, and identifying the underlying manifold as homogeneous informs the choice of Lie group and the design of the integrator~\cite{MuntheKaas97,MuntheKaas99a,%
Celledoni03,Malham08,MuntheKaas14,MuntheKaas15}.

\subsection*{Structure of the paper}

The main results are \cref{thm:cp_quotient,thm:tt_quotient,thm:tucker_quotient}.

\Cref{sec:Preliminaries} briefly recaps the concepts from three different fields of study---homogeneous spaces, multilinear algebra, and algebraic geometry---that we use to construct our homogeneous structure.
\Cref{sec:CP tensors,sec:Tucker tensors,sec:Tensor trains} apply these to sets of tensors with fixed CP rank, multilinear rank, and TT rank respectively.
The constructions are similar to each other, and the basic results are repeated without proof in \cref{sec:Tucker tensors,sec:Tensor trains} after having been introduced in detail in \cref{sec:CP tensors}.

\subsection*{Acknowledgements}

This research was funded by BOF project C16/21/002 by the Internal Funds KU Leuven and FWO project G080822N, and by FWO travel grant K243425N.

I am grateful to my PhD examination committee, who read and gave feedback on parts of this text that were in my thesis.
I am also grateful to Nick Vannieuwehoven (also committee member) and Tim Seynnaeve for fruitful discussions about algebraic geometry.

\section{Preliminaries}%
\label{sec:Preliminaries}

Let $d \geq 3$ and let $n_1$, \dots, $n_d$ be integers $\geq 2$.
If $\mathrm{GL}(n)$ denotes the real group of $n \times n$ invertible matrices, then
\begin{align}
	G = \mathrm{GL}(n_1) \times \dots \times \mathrm{GL}(n_d)
\end{align}
has a natural action on tensor on the set of tensors $\reals^{n_1} \otimes \dots \otimes \reals^{n_d}$ via
\begin{align}
	(g_1, \dots, g_d) \cdot (v_1 \otimes \dots \otimes v_d) = (g_1 v_1) \otimes \dots \otimes (g_d v_d).
\end{align}
This is the Lie group action we will use to construct our manifolds.

We will write $G \actson M$ for \enquote{the action of $G$ on $M$}.

\subsection{Homogeneous spaces}%
\label{sub:Homogeneous spaces}

A \emph{homogeneous space} is a quotient $A / B$ of Lie groups, along with some manifold structure that is compatible with the projection $\pi \from A \to A / B$, $a \mapsto a B$.
Lee~\cite[chapters 7 and 21]{Lee13} is a standard introduction to smooth homogeneous spaces and O'Neill~\cite[chapter 11]{Oneill83} is a standard introduction to Riemannian homogeneous spaces.

Fix an element $T \in \reals^{n_1} \otimes \dots \otimes \reals^{n_d}$ and define its \emph{orbit}
\begin{align}
	G \cdot T = \set{g \cdot T \suchthat g \in G},
\end{align}
and its \emph{stabilizer} (or \emph{isotropy subgroup})
\begin{align}
	H = \set{g \in G \suchthat g \cdot T = T}.
\end{align}
There is a unique smooth manifold structure on $G \cdot T = G / H$ such that, for all $f$, $f \from G \cdot T \to \reals$ is smooth if and only if $f \circ \pi \from G \to \reals$ is smooth.
Then $\pi$ is called a \emph{smooth submersion}.

Similarly, there is a natural Riemannian structure on $G \cdot T$ defined via $G / H$.
First, consider the Euclidean inner product on $G$'s algebra $\mathfrak{g}$, and translate this inner product to a right-invariant Riemannian metric on $G$.
Second, for every $g \in G$, we can demand that $\dd{\pi}_g \from T_g G \to T_{\pi(g)} (G \cdot T)$ preserves the length of vectors that are orthogonal to the fiber $\inverse{\pi}(\pi(g))$~\cite[section 1.2.2]{Petersen06}.
Such vectors are called \emph{horizontal}, and $\pi$ is then called a \emph{Riemannian submersion}.
The resulting metric on $G / H$ is called the \emph{canonical metric}.

Notably, geodesics in $G$ whose initial velocity is horizontal will keep having horizontal velocity, allowing us to define \emph{horizontal geodesics}.
From this also follows a one-to-one correspondence between horizontal geodesics in $G$ through $g$ and geodesics in $G \cdot T$ through $\pi(g)$~\cite[proposition 3.31]{Cheeger08}.
In terms of the manifold exponential, we can write
\begin{align}
	\operatorname{exp}_{\pi(g)}(\dd{\pi}_g X) = \pi(\operatorname{exp}_g(X))
\end{align}
when $X \in T_g G$ is horizontal.
This is useful because it allows us to lift geodesics in $G \cdot T$ to geodesics in $G$.

\subsection{Algebraic geometry}%
\label{sub:Algebraic geometry}

To derive quotients, we will need some results from algebraic geometry.
Landsberg~\cite{Landsberg12} is a good reference for the algebraic perspective on tensors.
Here, we just introduce a few concepts that will be useful later.
An \emph{Algebraic variety} is a space that is a solution set to some algebraic equation.
Especially, putting an upper bound on the rank of tensors is an algebraic condition and thus yields an algebraic variety.
The link between homogeneous spaces and algebraic geometry is that the orbits $G \cdot T$ are open and dense subsets of such varieties.
Particularly, they are open subsets in the \emph{Zariski topology}, where the closed sets are algebraic varieties.

\subsection{Multilinear algebra}%
\label{sub:Multilinear algebra}

Tensors are high-dimensional objects, and working with them in practice often requires using some \emph{decomposition}.
Kolda and Bader~\cite{Kolda09} survey the CP and Tucker decompositions, and Oseledets~\cite{Oseledets11} introduces tensor trains.
Graphically, tensor decompositions can be thought of as \emph{Penrose diagrams}.
For example, \cref{fig:svd} shows a compact SVD decomposition.
The idea is that when the dimensions of the inner edges are small enough, then the decomposed tensor is cheaper to store and to operate on.

\begin{figure}[h]
	\centering
	\includegraphics{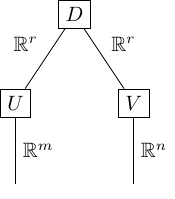}
	\caption{%
	Penrose diagram for the compact SVD decomposition, $U D \transpose{V}$, of an $m \times n$ rank $r$ matrix.
	Edges are labeled with the vector space they represent.%
	}
	\label{fig:svd}
\end{figure}

\subsubsection{CP tensors}%
\label{ssub:CP tensors}

\begin{definition}
	The \emph{CP rank} (often just \emph{rank}) of a tensor $T \in \reals^{n_1} \otimes \dots \otimes \reals^{n_d}$ is the smallest $r$ such that $T$ is a sum of $r$ rank $1$ tensors:
	\begin{align}\label{eq:cp_decomposition}
		T = \sum_{j = 1}^{r} v_{1}^{j} \otimes \dots \otimes v_{d}^{j},
	\end{align}
	where $v_{i}^{j} \in \reals^{n_i}$.
	Such an expression is called a \emph{CP decomposition}.
	It is illustrated in \cref{fig:cpd}.
\end{definition}

\begin{figure}[h]
	\centering
	\includegraphics{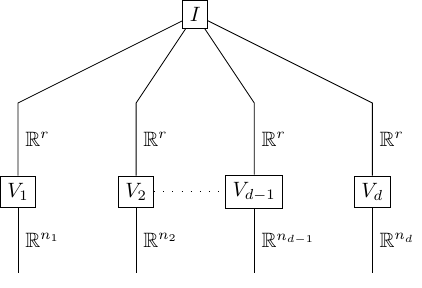}
	\caption{%
	Penrose diagram for the CP decomposition.
	Edges are labeled with the vector space they represent, and $V_i = \begin{bmatrix} v_{i}^{1} & \dots & v_{i}^{r} \end{bmatrix}$.
	$I$ is the diagonal tensor $I_{i_1 \dots i_d} = 1$ if $i_1 = \dots = i_d$ else $I_{i_1 \dots i_d} = 0$.%
	}
	\label{fig:cpd}
\end{figure}

\subsubsection{Tucker tensors}%
\label{ssub:Tucker tensors}

\begin{definition}
	The \emph{multilinear} (or \emph{Tucker}) rank of a tensor $T \in \reals^{n_1} \otimes \dots \otimes \reals^{n_d}$ is the tuple $(t_1, \dots, t_d)$ with
	\begin{align}\label{eq:tucker_rank}
		t_i = \operatorname{rank}(T \from \reals^{n_i} \to \reals^{n_1} \otimes \cdots \mathhat{\reals^{n_i}} \cdots \otimes \reals^{n_d}).
	\end{align}
\end{definition}

\begin{proposition}\label{prop:tucker_decomposition}
	If $T \in \reals^{n_1} \otimes \dots \otimes \reals^{n_d}$ has multilinear rank $(t_1, \dots, t_d)$ then $T$ can be written as
	\begin{align}\label{eq:tucker_decomposition}
		T_{k_1 \dots k_d} = \sum_{\alpha_1 = 1}^{t_1} \cdots \sum_{\alpha_d = 1}^{t_d} C_{\alpha_1 \dots \alpha_d} (G_1)_{k_1}{}^{\alpha_1} \cdots (G_d)_{k_d}{}^{\alpha_d}.
	\end{align}
	Such an expression is called a \emph{Tucker decomposition}.
	It is illustrated in \cref{fig:tucker_decomposition}.
\end{proposition}

\begin{proof}
	By \cref{eq:tucker_rank}, we can use the matrix rank decomposition in each mode to arrive at \cref{eq:tucker_decomposition}.
\end{proof}

Also, conversely, a tensor admitting a Tucker decomposition with inner dimensions $t_1$, \dots, $t_d$ has multilinear rank at most $(t_1, \dots, t_d)$.
If the inner dimensions and multilinear rank are equal, then the $G_i$ have maximal rank.

If the $G_i$ have maximal rank, we could without loss of generality require that they be orthogonal.
Many authors do this, but the stabilizer is easier to derive if we do not.

\begin{figure}[ht]
	\centering
	\includegraphics{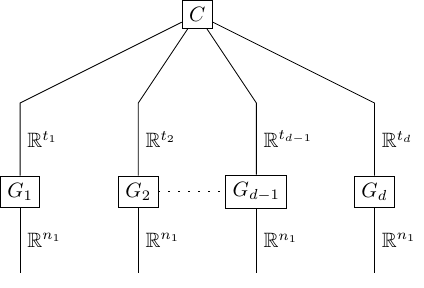}
	\caption{
	Penrose diagram for the Tucker decomposition \cref{eq:tucker_decomposition}.
	Edges are labeled with the vector space they represent.%
	}
	\label{fig:tucker_decomposition}
\end{figure}

\subsubsection{Tensor trains}%
\label{ssub:Tensor trains}

\begin{definition}
	The \emph{TT rank} of a tensor $T \in \reals^{n_1} \otimes \dots \otimes \reals^{n_d}$ is the tuple $(s_1, \dots, s_{d - 1})$ with
	\begin{align}\label{eq:tt_rank}
		s_i = \operatorname{rank}(T \from \reals^{n_1} \otimes \dots \otimes \reals^{n_i} \to \reals^{n_{i + 1}} \otimes \dots \otimes \reals^{n_d}).
	\end{align}
\end{definition}

\begin{proposition}[{Oseledets~\cite[Theorem 2.1]{Oseledets11}}]\label{prop:tt_decomposition}
	If $T \in \reals^{n_1} \otimes \dots \otimes \reals^{n_d}$ has TT rank $(s_1, \dots, s_{d - 1})$ then $T$ can be written as
	\begin{align}\label{eq:tt_decomposition}
		T_{k_1 k_2 \dots k_d} = \sum_{\alpha_1 = 1}^{s_1} \cdots \sum_{\alpha_{d - 1} = 1}^{s_{d - 1}} (F_1)_{k_1 \alpha_1} (F_2)^{\alpha_1}{}_{k_2 \alpha_2} \cdots (F_d)^{\alpha_{d - 1}}{}_{k_d}.
	\end{align}
	Such an expression is called a \emph{TT decomposition}.
	It is illustrated in \cref{fig:tt_decomposition}.
\end{proposition}

\begin{figure}[h]
	\centering
	\includegraphics{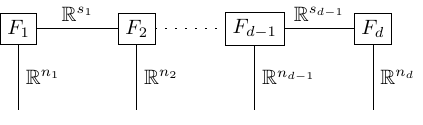}
	\caption{%
	Penrose diagram for the TT decomposition \cref{eq:tt_decomposition}.
	Edges are labeled with the vector space they represent.
	}
	\label{fig:tt_decomposition}
\end{figure}

Note also that a converse of \cref{prop:tt_decomposition} is true, that a TT decomposition with inner dimensions $s_1$, \dots, $s_{d - 1}$ has TT rank at most $(s_1, \dots, s_{d - 1})$.

\section{The CP manifold}%
\label{sec:CP tensors}

\begin{proposition}[Kruskal's theorem]\label{prop:kruskal}
	Let $T \in \reals^{n_1} \otimes \dots \otimes \reals^{n_d}$ be CP rank $r$ with CP decomposition \cref{eq:cp_decomposition}.
	Define $k_i$, for each mode $i$, so that any size $k_i$ subset of the factors $v_{i}^{1}$, \dots, $v_{i}^{r}$ is independent.
	If
	\begin{align}\label{eq:kruskals_condition}
		2 r + d - 1 \leq \sum_{i = 1}^{d} k_i,
	\end{align}
	then the decomposition is unique up to permutation of the terms and rescaling the factors.
\end{proposition}

Landsberg~\cite[Theorem 12.5.3.2]{Landsberg12} formulates Kruskal's theorem for complex vector spaces, but if a decomposition is real and complex unique then it is of course also real unique.
The condition \cref{eq:kruskals_condition} is called \emph{Kruskal's criterion}.
Whenever we use \nameref{prop:kruskal} in the paper, we are using it via the following corollary.

\begin{corollary}
	Let $T \in \reals^{n_1} \otimes \dots \otimes \reals^{n_d}$ be CP rank $r$ with CP decomposition \cref{eq:cp_decomposition}.
	Assume that, for each mode $i$, the factors $v_{i}^{1}$, \dots, $v_{i}^{r}$ are linearly independent.
	Then the decomposition is unique up to permutation of the terms and rescaling of the factors.
\end{corollary}

Note that $v_{i}^{1}$, \dots, $v_{i}^{r}$ are linearly independent in $\reals^{n_i}$ implies $r \leq n_i$ for all $i$.

\subsection{Smooth manifold}

Let $\inverse{\operatorname{rank}}(r)$ denote the set of tensors with CP rank $r$.
Its closure, $\mathbar{\inverse{\operatorname{rank}}(r)}$, is an algebraic variety.

\begin{lemma}\label{prop:cp_tensor_open_orbit}
	Let $r \leq n_i$ for all $i$.
	Then
	\begin{align}
		G \actson \inverse{\operatorname{rank}}(r)
	\end{align}
	has an open dense orbit, $\Sigma_r$, consisting of elements with maximal multilinear rank.
\end{lemma}

\begin{remark}
	Maximal multilinear rank in $\inverse{\operatorname{rank}}(r)$ is $(r, \dots, r)$.
\end{remark}

\begin{proof}
	CP rank is preserved by the action of $G$ since
	\begin{enumerate}
		\item if $\sum_{j = 1}^{r} v_{1}^{j} \otimes \dots \otimes v_{d}^{j}$ is a rank $r$ decomposition of $T$, then $\sum_{j = 1}^{r} (g_1 v_{1}^{j}) \otimes \dots \otimes (g_d v_{d}^{j})$ is a rank $r$ decomposition of $g \cdot T$, so $\operatorname{rank}(T) \geq \operatorname{rank}(g \cdot T)$, and
		\item if $\sum_{j = 1}^{r} v_{1}^{j} \otimes \dots \otimes v_{d}^{j}$ is a rank $r$ decomposition of $g \cdot T$, then $\sum_{j = 1}^{r} (\inverse{g_1} v_{1}^{j}) \otimes \dots \otimes (\inverse{g_d} v_{d}^{j})$ is a rank $r$ decomposition of $T$, so $\operatorname{rank}(g\cdot T) \geq \operatorname{rank}(T)$.
	\end{enumerate}
	Thus the action is well-defined.

	Let $e_{i}^{1}$, \dots, $e_{i}^{n_{i}}$ be the standard basis for $\reals^{n_i}$ and define
	\begin{align}\label{eq:cp_fixed_point}
		T = \sum_{j = 1}^{r} e_{1}^{j} \otimes \dots \otimes e_{d}^{j}.
	\end{align}
	We now want to show that $\Sigma_r = G \cdot T$ is Zariski open in $\mathbar{\inverse{\operatorname{rank}}(r)}$.
	It then follows that $\Sigma_{r}$ is open and dense in $\inverse{\operatorname{rank}}(r)$.

	Any other tensor,
	\begin{align}
		S = \sum_{j = 1}^{r} f_{1}^{j} \otimes \dots \otimes f_{d}^{j},
	\end{align}
	with linearly independent $f_{i}^{1}$, \dots, $f_{i}^{r}$ is reached by the group element $(g_1, \dots, g_d)$ satisfying $g_i e_{i}^{j} = f_{i}^{j}$.
	Moreover, if $f_{i}^{1}$, \dots, $f_{i}^{r}$ are linearly dependent for some $i$, then $S$ can't be reached by $G$.
	The complement of $\Sigma_r$ is hence described by a Zariski closed condition, so $\Sigma_r$ is Zariski open.

	To show that elements in $\Sigma_r$ have maximal multilinear rank, note that the matrix
	\begin{align}
		T_{(i)} = \sum_{j = 1}^{r} (e_{1}^{j} \otimes \cdots \mathhat{e_{i}^{j}} \dots \otimes e_{d}^{j}) \otimes e_{i}^{j}
	\end{align}
	has rank $r$ since $e_{1}^{j} \otimes \cdots \mathhat{e_{i}^{j}} \cdots \otimes e_{d}^{j}$ are $r$ linearly independent vectors.
	Moreover, the action of $G$ preserves multilinear rank.
	
	To show that elements outside of $\Sigma_r$ do not have maximal multilinear rank, note that if $f_{i}^{1}$, \dots, $f_{i}^{r}$ are linearly dependent for some $i$, and
	\begin{align}
		R = \sum_{j = 1}^{r} f_{1}^{j} \otimes \dots \otimes f_{d}^{j},
	\end{align}
	then the column space of $R_{(i)}$ is spanned by $r - 1$ vectors.
\end{proof}

We now consider the subgroup of $G$ that fixes the $T$ defined in \cref{eq:cp_fixed_point}.
Applying \nameref{prop:kruskal}, we have the following result.

\begin{lemma}\label{prop:cp_tensor_stabilizer}
	The stabilizer $H$ of $G \actson \Sigma_r$ consists of elements of the form
	\begin{align}\label{eq:cp_tensor_stabilizer}
		h =
		\begin{bmatrix}
			 D_{1} Q & M_{1}\\
			& A_{1}
		\end{bmatrix}
		\times \dots \times
		\begin{bmatrix}
			 D_{d} Q & M_{d}\\
			& A_{d}
		\end{bmatrix},
	\end{align}
	where the $D_{i}$ are diagonal invertible $r \times r$ matrices such that $D_{1} \cdots D_{d} = 1$, $Q$ is a permutation matrix, the $A_{i}$ are invertible $(n_{i} - r) \times (n_{i} - r)$ matrices, and the $M_{i}$ are arbitrary $r \times (n_{i} - r)$ matrices.
\end{lemma}

Combining \cref{prop:cp_tensor_open_orbit,prop:cp_tensor_stabilizer}, we have the main result of this subsection.

\begin{theorem}\label{thm:cp_quotient}
	The set of tensors with CP rank $r$ and multilinear rank $(r, \dots, r)$ is a smooth homogeneous manifold,
	\begin{align}
		\Sigma_{r} = G / H.
	\end{align}
\end{theorem}

\subsection{Representatives}%
\label{sub:cp_representatives}

A point $p = g H \in \Sigma_{r}$ can be defined by specifying $g = (g_1, \dots, g_d)$.
However, this requires specifying $n_1^{2} + \dots + n_d^{2}$ numbers, while the dimension of $\Sigma_{r}$ is only $(n_1 + \dots + n_d) r - (d - 1) r$.
To address this, we observe that a generic $g_i \in \mathrm{GL}(n_i)$ can be reduced to block lower triangular form by the action of $H$.
Define $h_{i}$ via
\begin{align}
	g_{i} =
	\begin{bmatrix}
		g_{11} & g_{12}\\
		g_{21} & g_{22}
	\end{bmatrix}
	=
	\begin{bmatrix}
		g_{11} & \\
		g_{21} & 1
	\end{bmatrix}
	\underbrace{
	\begin{bmatrix}
		1 & \inverse{g_{11}} g_{12}\\
		 & g_{22} - g_{21} \inverse{g_{11}} g_{12}
	\end{bmatrix}
	}_{\inverse{h_i}},
\end{align}
and note that $h = (h_1, \dots, h_d) \in H$.
Here, we see that $g_{11}$ needs to be invertible.
If $g_{11}$ is not invertible, there is a permutation matrix $P$ such that $g_i' = P g_i$ has an invertible block $g'_{11}$.
In this way we only need to specify $g_{11}$, $g_{21}$, and $P$ for each $i$.

In practice, we want to choose $P$ so that the determinant of $g_{11}$ is maximized.
This is known as the \emph{submatrix selection} problem.
It has been studied in depth because of its application to CUR-type matrix decompositions.
See for example the review paper by Halko, Martinsson, and Tropp~\cite{Halko11}.
We also mention the recent paper by Osinsky~\cite{Osinsky25} showing that a quasioptimal choice can be made using $\order{n r^{2}}$ basic operations, and the randomized version proposed by Cortinovis and Kressner~\cite{Cortinovis25}.
Hence there are efficient algorithms to choose $g_{11}$.

\subsection{Riemannian manifold}%
\label{sub:cp_riemannian_manifold}

$G$'s algebra, $\mathfrak{g}$, consists of elements $Z = (Z_{1}, \dots, Z_{d})$ where the $Z_i$ are arbitrary $n_i \times n_i$ matrices.
We consider the Euclidean inner product on $\mathfrak{g}$, defined as
\begin{align}
	\innerproduct{Z}{Z'} ={}& \operatorname{tr}\left( Z_{1} \transpose{{Z'_{1}}} \right) + \dots + \operatorname{tr}\left( Z_{d} \transpose{{Z'_{d}}} \right).
\end{align}

$H$'s Lie algebra, $\mathfrak{h}$, is a Lie subalgebra of $\mathfrak{g}$ and by taking the derivative of the expression \cref{eq:cp_tensor_stabilizer} in \cref{prop:cp_tensor_stabilizer} we see that $\mathfrak{h}$ consists of elements $Y = (Y_{1}, \dots, Y_{d})$ where the $Y_{i}$ are on block form $\begin{bmatrix} Y^{i}_{11} & Y^{i}_{12}\\ & Y^{i}_{22} \end{bmatrix}$ with $Y^{i}_{11}$ diagonal such that $Y^{1}_{11} + \dots + Y^{d}_{11} = 0$, and $Y^{i}_{12}$ and $Y^{i}_{22}$ arbitrary.
The orthogonal complement of $\mathfrak{h}$, which we denote $\mathfrak{m}$, thus consists of elements $X = (X_{1}, \dots, X_{d})$ where the $X_{i}$ are on block form $\begin{bmatrix} X^{i}_{11} & \\ X^{i}_{21} & 0 \end{bmatrix}$ such that $X^{i}_{11}$ has the same diagonal for every $i$ but is otherwise arbitrary, and $X^{i}_{21}$ is arbitrary.

For any $g \in G$, the coset $g H$ is a submanifold of $G$.
Its tangent space at $g$ is called the \emph{vertical space}, and is denoted $\mathcal{V}_{g}$.
The orthogonal complement to the vertical space is called \emph{horizontal space}, and is denoted $\mathcal{H}_{g}$.
Thus $\mathfrak{h}$ and $\mathfrak{m}$ are the vertical and horizontal spaces at $1$.

Now, consider the right-invariant metric\footnote{We need the right-invariant metric rather than the left-invariant because we are dividing $G$ by $H$ on the right, so the metric on $G$ needs to be at least right-$H$-invariant.} on $G$ induced by the inner product on $\mathfrak{g}$.
The \emph{canonical metric} on $\Sigma_r$ is then defined by demanding that the quotient map $\pi \from G \to \Sigma_r$ is a Riemannian submersion.
By construction, $\dd\pi_{g}\eval_{\mathcal{H}_{g}} \from \mathcal{H}_{g} \to T_{g H} \Sigma_r$ is a linear isomorphism.
So in other words, we are defining our metric by demanding that $\dd\pi_{g}\eval_{\mathcal{H}_{g}}$ is also an isometry.

Note also that the right-invariant metric on $\mathrm{GL}(n)$ is left-$\mathrm{O}(n)$-invariant.
In light of our discussion in \cref{sub:cp_representatives}, this is important because permutation matrices are orthogonal and so we can work with any representative.

\bigskip

We now want to derive a more explicit description of the horizontal space.
First, note that $\mathcal{V}_{g} = g \mathfrak{h}$, so that the vertical space consists of elements $Y = (Y_{1}, \dots, Y_{d})$ where the $Y_{i} = g_i \begin{bmatrix} Y_{11} & Y_{12} \\ & Y_{22} \end{bmatrix}$.
From now on, we suppress the $i$ superscript, but note that $Y_{11} = Y^{i}_{11}$ and the other blocks still depends on $i$.
Second, if we choose a block lower triangular representative as in \cref{sub:cp_representatives}, we have
{\hfuzz=8pt
\begin{align}
	&\innerproduct{X_{i}}{Y_{i}}_{g_{i}} =
		\operatorname{tr} \left(
		\begin{bmatrix}
			X_{11} & X_{12}\\
			X_{21} & X_{22}
		\end{bmatrix}
		\inverse{\begin{bmatrix}
			g_{11} & \\
			g_{21} & 1
		\end{bmatrix}}
		\transpose{\left( 
		\begin{bmatrix}
			g_{11} & \\
			g_{21} & 1
		\end{bmatrix}
		\begin{bmatrix}
			Y_{11} & Y_{12}\\
			& Y_{22}
		\end{bmatrix}
		\inverse{\begin{bmatrix}
			g_{11} & \\
			g_{21} & 1
		\end{bmatrix}}
		\right)}
	\right)\\
	={}&
		\operatorname{tr} \left(
		\begin{bmatrix}
			X_{11} & X_{12}\\
			X_{21} & X_{22}
		\end{bmatrix}
		\begin{bmatrix}
			\inverse{g_{11}} g_{11}^{-\mathsf{T}} & -\inverse{g_{11}} g_{11}^{-\mathsf{T}} \transpose{g_{21}}\\
			-g_{21} \inverse{g_{11}} g_{11}^{-\mathsf{T}} & 1 + g_{21} \inverse{g_{11}} g_{11}^{-\mathsf{T}} \transpose{g_{21}}
		\end{bmatrix}
		\transpose{\begin{bmatrix}
			g_{11} Y_{11} & g_{11} Y_{12}\\
			g_{21} Y_{11} & g_{21} Y_{12} + Y_{22}
		\end{bmatrix}}
	\right)
\end{align}
}
Since the $Y_{12}$ is arbitrary, $Y'_{12} = g_{11} Y_{12}$ is also arbitrary.
Similarly, $Y'_{22} = g_{21} Y_{12} + Y_{22}$ is arbitrary.
Collecting the $Y'_{12}$ coefficients yields the condition
\begin{align}
	-X_{11} \inverse{g_{11}} g_{11}^{-\mathsf{T}} \transpose{g_{21}} + X_{12} (1 + g_{21} \inverse{g_{11}} g_{11}^{-\mathsf{T}} \transpose{g_{21}}) = 0.
\end{align}
Note that $(1 + g_{21} \inverse{g_{11}} g_{11}^{-\mathsf{T}} \transpose{g_{21}})$ is invertible since it is the sum of a positive definite and positive semidefinite matrix, and so it is positive definite.
We can hence solve for $X_{12}$:
\begin{align}
	X_{12} = X_{11} \Gamma_{12},
\end{align}
where $\Gamma_{12} = \inverse{g_{11}} g_{11}^{-\mathsf{T}} \transpose{g_{21}} \inverse{(1 + g_{21} \inverse{g_{11}} g_{11}^{-\mathsf{T}} \transpose{g_{21}})}$.
Similarly, collecting the $Y'_{22}$ coefficients and solving for $X_{22}$ yields
\begin{align}
	X_{22} = X_{21} \Gamma_{12}.
\end{align}
We also note that this implies that, for each $i$,
\begin{align}\label{eq:horizontal_rank_decomposition}
	X_{i} ={} \begin{bmatrix}
		X_{11} & X_{11} \Gamma_{12}\\
		X_{21} & X_{21} \Gamma_{12}
	\end{bmatrix}
	={} \begin{bmatrix}
		X_{11}\\
		X_{21}
	\end{bmatrix}
	\begin{bmatrix}
		1 & \Gamma_{12}
	\end{bmatrix}
\end{align}
is at most rank $r$.

\paragraph{The power series trick}
Before collecting the $Y_{11}$ coefficients, we discuss an alternative expression for $\Gamma_{12}$ that will allow us to compute it more efficiently.
As it is currently written, building $\Gamma_{12}$ requires inverting the $n_{i} \times n_{i}$ matrix $1 + g_{21} \inverse{g_{11}} g_{11}^{-\mathsf{T}} \transpose{g_{21}}$.
Note however that this inverse is an analytic function of the rank $r$ matrix $g_{21} \inverse{g_{11}} g_{11}^{-\mathsf{T}} \transpose{g_{21}}$.
Its power series is
\begin{align}
	\frac{1}{1 + x} = 1 - x + x^{2} - x^{3} + \cdots.
\end{align}
If $M = A B$ is a rank $r$ decomposition of an $n \times n$ matrix $M$, then
\begin{align}
	\frac{1}{1 + M} ={}& 1 - A B + (A B)^{2} - (A B)^{3} + \cdots\\
	={}& 1 + A (-1 + B A - (B A)^{2} + \cdots) B\\
	={}& 1 - A \frac{1}{1 + B A} B.
\end{align}
But $B A$ is an $r \times r$ matrix, allowing $-1 / (1 + B A)$ to be evaluated efficiently.
This trick is also useful later for evaluating the matrix exponential.
For $\Gamma_{12}$, we have the following expression,
\begin{align}\label{eq:Gamma12}
	\Gamma_{12} ={}& \inverse{g_{11}} \left( 1 - \frac{\left[ g_{11}^{-\mathsf{T}} \transpose{g_{21}} g_{21} \inverse{g_{11}} \right]}{1 + \left[ g_{11}^{-\mathsf{T}} \transpose{g_{21}} g_{21} \inverse{g_{11}} \right]} \right) g_{11}^{-\mathsf{T}} \transpose{g_{21}}.
\end{align}
In particular, the expression in brackets is an $r \times r$ matrix.

We are now ready to collect the $Y_{11}$ coefficients.
We get the condition that the $k$th column of
\begin{align}
	&\begin{bmatrix}
		X_{11}\\
		X_{21}
	\end{bmatrix}
	\inverse{g_{11}} g_{11}^{-\mathsf{T}}
	- \begin{bmatrix}
		X_{12}\\
		X_{22}
	\end{bmatrix}
	g_{21} \inverse{g_{11}} g_{11}^{-\mathsf{T}}
	= \begin{bmatrix}
		X_{11}\\
		X_{21}
	\end{bmatrix}
	\left(
	\inverse{g_{11}} g_{11}^{-\mathsf{T}}
	-
	\Gamma_{12}
	g_{21} \inverse{g_{11}} g_{11}^{-\mathsf{T}}
	\right)
\end{align}
has the same dot product with the $k$th column of $\begin{bmatrix} g_{11}\\ g_{12} \end{bmatrix}$ for every $i$.

This completes the description of the horizontal space $\mathcal{H}_{g}$.
The most important takeaway is that horizontal vectors have a decomposition \cref{eq:horizontal_rank_decomposition}.

\subsection{Riemannian homogeneous manifold}%
\label{sub:cp_riemannian_homogeneous_manifold}

The construction described in \cref{sub:cp_riemannian_manifold} uses a right-invariant metric on $G$, but the resulting metric on $\Sigma_r$ is not necessarily right-invariant.
In fact, not all smooth homogeneous manifolds allow for an invariant metric.
The underlying issue is that since $h \cdot p = p$ for all $h \in H$, the inner product on $T_p \Sigma_r$ has to be invariant under $H \actson \Sigma_r$, but such an inner product might not exist.

The technical condition that we want to satisfy is that there exists a subspace $\mathfrak{p} \subset \mathfrak{g}$, with $\mathfrak{p} \oplus \mathfrak{h} = \mathfrak{g}$, such that $h \mathfrak{p} \inverse{h} \subset \mathfrak{p}$ for all $h \in H$.
$G / H$ is then called \emph{reductive} and $\mathfrak{p}$ is called an \emph{invariant subspace}.
See O'Neill~\cite[Chapter 11]{Oneill83} for more details.

\begin{proposition}\label{prop:cp_reductive}
	$\Sigma_r$ is reductive if and only if $r = n_1 = \dots = n_d$.
	Moreover, when $\Sigma_r$ is reductive, the $\mathfrak{m}$ that we defined in \cref{sub:cp_riemannian_manifold} is an invariant subspace.
\end{proposition}

\Cref{prop:cp_reductive} is a higher-order version of \cite[Proposition 3.4]{Vandereycken12} and \cite[Proposition 5.7]{MuntheKaas15}, which say that fixed-rank matrices are only reductive when they are square and full rank.

\begin{proof}
	Assume $r = n_1 = \dots = n_d$ and let $h = D_1 Q \times \dots \times D_d Q \in H$.
	Let $X \in \mathfrak{m}$ and consider the expression
	\begin{align}
		h X \inverse{h} = (D_1 Q X_1 \inverse{Q} \inverse{D_1}, \dots, D_d Q X_d \inverse{Q} \inverse{D_d}).
	\end{align}
	If $\sigma$ denotes the permutation that corresponds to $Q$, then on the diagonals we have that
	\begin{align}
		(D_i Q X_i \inverse{Q} \inverse{D_i})_{kk} ={}& (D_i)_{kk} (Q X_i \inverse{Q})_{kk} (\inverse{D_i})_{kk}\nonumber\\
		={}& (D_i)_{kk} (X_i)_{\sigma(k) \sigma(k)} (\inverse{D_i})_{kk}\nonumber\\
		={}& (X_i)_{\sigma(k) \sigma(k)}.
	\end{align}
	Hence $h X \inverse{h}$ is contained in $\mathfrak{m}$, which is what we wanted to show.

	Conversely, assume there exists an invariant subspace $\mathfrak{p}$ when $n_i > r$ for some $i$.
	Note that elements in $\mathfrak{p}$ are determined by their projection to $\mathfrak{m}$.
	Concretely, elements in $\mathfrak{p}$ are determined by their $r$ first columns.
	So consider an element $X = (X_1, \dots, X_d) \in \mathfrak{p}$ with
	\begin{align}
		X_i =
		\begin{bmatrix}
			X_{11} & X_{12}\\
			X_{21} & X_{22}
		\end{bmatrix}.
	\end{align}
	$X_{12}$ and $X_{22}$ are functions of $X_{11}$ and $X_{21}$.
	For our contradiction, we will show that $X_{11}$ is a multiple of the identity.
	Then the dimension of $\mathfrak{p}$ is too low to be complementary to $\mathfrak{h}$.

	First, assume $X_{21} = 0$ and let $h = \begin{bmatrix} 1 & M\\ & 1 \end{bmatrix}$ with $M$ arbitrary.
	Then $\inverse{h} = \begin{bmatrix} 1 & -M\\ & 1 \end{bmatrix}$ and
	\begin{align}
		h X_i \inverse{h}
		={}& 
		\begin{bmatrix}
			X_{11} & -X_{11} M + X_{12} + M X_{22}\\
			& X_{22}
		\end{bmatrix}\\
		\inverse{h} X_i h
		={}& 
		\begin{bmatrix}
			X_{11} & X_{11} M + X_{12} - M X_{22}\\
			& X_{22}
		\end{bmatrix}.
	\end{align}
	The first $r$ columns of $h X_i \inverse{h}$ and $\inverse{h} X_i h$ are the same, so by our previous comment they must be equal.
	Thus
	\begin{align}
		X_{11} M - M X_{22} = 0
	\end{align}
	for all $M$.
	Writing this as
	\begin{align}
		(X_{11} \otimes 1 - 1 \otimes X_{22}) \operatorname{vec} M = 0,
	\end{align}
	it is clear that the only solutions are when $X_{11}$ and $X_{22}$ are multiples of identity matrices and have the same norm.

	Second, if $X_{21} \neq 0$, we can use the same argument on
	\begin{align}
		\begin{bmatrix}
			X_{11} & X_{12}\\
			X_{21} & X_{22}
		\end{bmatrix}
		-
		\begin{bmatrix}
			& X'_{12}\\
			X_{21} & X'_{22}
		\end{bmatrix}
		\in \mathfrak{p}
	\end{align}
	to see that $X_{11}$ again is a multiple of the identity.

	Whenever $r \geq 2$, the above is enough for a contradiction.
	However, when $r = 1$, $X_{11}$ is just a $1 \times 1$ matrix, and so showing that it is a multiple of the identity yields no contradiction.
	We need to change the argument slightly.
	Consider now instead the case $X_{11} = 0$.
	\begin{align}
		h
		X_i
		\inverse{h} ={}&
		\begin{bmatrix}
			M X_{21} & *\\
			X_{21} & X_{22} - X_{21} M
		\end{bmatrix}.
	\end{align}
	By the previous paragraph, we are free to add and subtract any multiple of the identity matrix and still stay in $\mathfrak{p}$.
	If we subtract the real number $M X_{21}$ from the diagonal, we are left with
	\begin{align}
		\begin{bmatrix}
			& *\\
			X_{21} & X_{22} - X_{21} M - M X_{21} \cdot 1
		\end{bmatrix}.
	\end{align}
	Similarly to before, since the first column is the same as $X_i$, it must be equal to $X_i$, but for this to be true for all $M$ implies $X_{21} = 0$.
	$X_{21}$ is ours to choose, so any restriction on it is a contradiction.
\end{proof}

\subsection{Geodesics}
\label{sub:cp_geodesics}

Since the subgroup associated with $Q$ is discrete, we may for the purposes of this subsection ignore it and set $Q = 1$.

Let $g_i \in \mathrm{GL}(n_i)$ and $X_i \in T_{g_i} \mathrm{GL}(n_i)$.
Andruchow, Larotonda, Recht, and Varela \cite{Andruchow14} show that the geodesics on the general linear group are\footnote{Note that they use a left-invariant metric while we use a right-invariant.}
\begin{align}\label{eq:geodesics_on_GL(n)}
	\operatorname{exp}_{g_i}(X_i) = \operatorname{mexp}(X_i \inverse{g_i} - \transpose{(X_i \inverse{g_i})}) \operatorname{mexp}(\transpose{(X_i \inverse{g_i})}) g_{i}.
\end{align}

Geodesics on $\Sigma_r = G / H$ are images of horizontal geodesics in $G$ under the quotient map $\pi \from G \to G / H$.
In this subsection, assuming $r \ll n_{i}$, our aim is to efficiently compute (an element in the same equivalence class as) \cref{eq:geodesics_on_GL(n)} when $X_i$ is part of a horizontal vector.
We are going to show that this can be done in $\order{n_{i} r^{2}}$ basic operations.
This is a considerable improvement over computing the matrix exponentials naively, which is $\order{n_{i}^{3}}$ basic operations.

First, recall that horizontal vectors are on the form \cref{eq:horizontal_rank_decomposition}.
We thus have a rank $r$ decomposition
\begin{align}\label{eq:horizontal_rank_decomposition2}
	X_{i} \inverse{g_{i}} = \begin{bmatrix}
		X_{11}\\
		X_{21}
	\end{bmatrix}
	\begin{bmatrix}
		\inverse{g_{11}} - \Gamma_{12} g_{21} \inverse{g_{11}} & \Gamma_{12}
	\end{bmatrix}.
\end{align}
To compute the matrix exponential of such a vector, we can use the same power series trick as before.
If we name the factors in \cref{eq:horizontal_rank_decomposition2} $A \in \reals^{n_{i} \times r}$ and $B \in \reals^{r \times n_{i}}$ respectively, then
\begin{align}
	\operatorname{mexp}(X_{i} \inverse{g_{i}}) ={}& 1 + A B + \frac{1}{2} (A B)^{2} + \dots\nonumber\\
	={}& 1 + A \psi_{1}\left( B A \right) B,
\end{align}
where $\psi_{1}(x) = 1 + \frac{1}{2} x + \frac{1}{3!} x^{2} + \cdots$ is the Taylor series for $(\exp{x} - 1) / x$.
The notation $\psi_{1}$ comes from the theory of exponential integrators, where the functions
\begin{align}
	\psi_{k}(x) = \sum_{j = 0}^{\infty} \frac{1}{(j + k)!} x^{j}
\end{align}
are known as \emph{the $\psi$ functions}.
Importantly, the argument to $\psi_{1}$ is an $r \times r$ matrix, so it can be evaluated cheaply.
We will discuss exactly how shortly.

Second, we have a rank $2 r$ decomposition
\begin{align}
	X_i \inverse{g_{i}} - \transpose{(X_i \inverse{g_{i}})}
	= \begin{bmatrix}
		X_{11} & -\transpose{(\inverse{g_{11}} - \Gamma_{12} g_{21} \inverse{g_{11}})}\\
		X_{21} & -\transpose{\Gamma_{12}}
	\end{bmatrix}
	\begin{bmatrix}
		\inverse{g_{11}} - \Gamma_{12} g_{21} \inverse{g_{11}} & \Gamma_{12}\\
		\transpose{X_{11}} & \transpose{X_{21}}
	\end{bmatrix}.
\end{align}
So, similarly to before, denoting the factors $A' \in \reals^{n_{i} \times 2 r}$ and $B' \in \reals^{2 r \times n_{i}}$ respectively,
\begin{align}
	&\operatorname{mexp}(X_i \inverse{g_i} - \transpose{(X_i g_i)}) = 1 + A' \psi_{1}(B' A') B'.
\end{align}
Here, the argument to $\psi_{1}$ is an $2 r \times 2 r$ matrix, so it can be evaluated cheaply.

Putting all this into \cref{eq:geodesics_on_GL(n)} and doing the multiplications, we find an expression for the first $r$ columns,
\begin{align}
	&\operatorname{exp}_{g_{i}}(X_i) \begin{bmatrix} 1 \\ 0 \end{bmatrix}\nonumber\\
	&= \begin{bmatrix} g_{11}\\ g_{21} \end{bmatrix}
	+ \transpose{B} \transpose{\psi_{1}(B A)} \transpose{A} \begin{bmatrix} g_{11}\\ g_{21} \end{bmatrix}
	+ A' \psi_{1}(B' A') B' \begin{bmatrix} g_{11}\\ g_{21} \end{bmatrix}
	+ \transpose{B} \transpose{\psi_{1}(B A)} \transpose{A} A' \psi_{1}(B' A') B' \begin{bmatrix} g_{11}\\ g_{21} \end{bmatrix}
	\label{eq:geodesics_efficient}
\end{align}
Advantageously, if the multiplications are done in the right order this does not require forming any $n_i \times n_i$ matrices.

\bigskip

We now return to how to compute $\psi_{1}$.
We will do it similarly to how the matrix exponential is usually computed, using \emph{Padé approximation} and \emph{scaling and squaring}.
See Moler and van Loan~\cite[methods 2 and 3]{Moler03} and Higham~\cite[sections 10.3 and 10.7.4]{Higham08}.
Let $M$ be an $r \times r$ matrix and assume first that $\norm{M} \leq 1 / 2$.
Then $\psi_{1}(M)$ is approximated to double precision from its degree $(6, 6)$ Padé approximant~\cite[theorem 10.31]{Higham08}
\begin{align}\label{eq:pade_approximant}
	r_{66}(M) = \frac{
		1 + M / 26 + 5 M^2 / 156 + M^3 / 858 + M^4 / 5720 + M^5 / 205920 + M^6 / 8648640
	}{
		1 - 6 M / 13 + 5 M^2 / 52 - 5 M^3 / 429 + M^4 / 1144 - M^5 / 25740 + M^6 / 1235520
	}.
\end{align}
We prove this in \cref{lemma:pade_approximant}.

On the other hand, if $\norm{M} > 1 / 2$, then we do not use the Padé approximant directly.
Let $z$ be an integer such that $\norm{M} 2^{-z} \leq 1 / 2$.
Keeping in mind that these expressions are only shorthands for their Taylor series, we have
\begin{align}
	\psi_{1}(M) ={}& \frac{\operatorname{mexp}(M) - 1}{M}\nonumber\\
	={}& \frac{\operatorname{mexp}(2^{-z} M)^{2^z} - 1}{M}\nonumber\\
	={}& \frac{\operatorname{mexp}(2^{-z} M) - 1}{M} (\operatorname{mexp}(2^{-z} M) + 1) \cdots (\operatorname{mexp}(2^{-1} M) + 1)\nonumber\\
	={}& 2^{-z} \psi_{1}(2^{-z} M) (\operatorname{mexp}(2^{-z} M) + 1) \cdots (\operatorname{mexp}(2^{-1} M) + 1).%
	\label{eq:scaling_and_squaring}
\end{align}
In the second to last step, we used $x^2 - 1 = (x - 1) (x + 1)$ recursively.
This scaling and squaring step is similar to the algorithm proposed by Hochbruck, Lubich, and Selhofer~\cite{Hochbruck98}.

\paragraph{Counting the operations}
We now count the number of basic operations required to evaluate \cref{eq:geodesics_efficient}.
We use the same conventions as \cite[Table C.1]{Higham08}, where a $(a \times b) \times (b \times c)$ matrix multiplication is $2 a b c$ basic operations, and a $(a \times b) \times (b \times c)$ matrix division is $8 a b c / 3$ basic operations.
Terms that are $\order{n_i r}$ or $\order{r^{2}}$, such as adding $r \times r$ matrices, are ignored.

\begin{proposition}\label{prop:cp_efficient_geodesics}
	Given a tangent vector $X \in T_{p} \Sigma_r$, $\operatorname{exp}_{p}(X)$ can be estimated\footnote{%
		Meaning that everything is computed exactly except for $\psi_{1}$, which is Padé approximated to within double precision.
	}
	using
	\begin{align}
			\sum_{i = 1}^{d} \left[ \frac{110}{3} n_i r^{2} + (146 + 36 z_i) r^{3} + \order{n_i r + r^{2}} \right] \textrm{ basic operations}
	\end{align}
	where\footnote{This formula is valid for any norm, as long as it is the same as in \cref{sec:Error bound for Pade approximant}.} $z_i = \lceil \operatorname{log}_2 \norm{X_{i} \inverse{g_{i}}} \rceil + 2$.
\end{proposition}

See \cref{sec:Appendix: Counting the operations} for the proof.

This can be viewed as a tensorial version of Vandereycken's et al.~\cite{Vandereycken12} corollary 4.2.
We also mention that the effects of rounding errors in the Padé approximant and squaring step are not completely understood~\cite{Moler03}, and so we do not do a full error analysis.

\section{The Tucker manifold}%
\label{sec:Tucker tensors}

In \cref{prop:tucker_decomposition}, we used the matrix rank decomposition recursively, and since that decomposition is unique up to a change of basis in the inner vector space, we immedeately arrive at the following.

\begin{proposition}\label{prop:tucker_uniqueness}
	The Tucker decomposition is unique up to a change of basis in $\reals^{t_1}$, \dots, $\reals^{t_d}$.
	More precisely,
	\begin{align}
		\sum_{\alpha_1 = 1}^{t_1} \cdots \sum_{\alpha_d = 1}^{t_d} C'_{\alpha_1 \dots \alpha_d} (G'_1)_{k_1}{}^{\alpha_1} \cdots (G'_d)_{k_d}{}^{\alpha_d} =
		\sum_{\alpha_1 = 1}^{t_1} \cdots \sum_{\alpha_d = 1}^{t_d} C_{\alpha_1 \dots \alpha_d} (G_1)_{k_1}{}^{\alpha_1} \cdots (G_d)_{k_d}{}^{\alpha_d}
	\end{align}
	iff there are matrices $U_1 \in \mathrm{GL}(t_1)$, \dots, $U_d \in \mathrm{GL}(t_d)$ such that
	\begin{align}
		(G'_i)_{k_i}{}^{\alpha_i} ={}& \sum_{\beta = 1}^{t_i} (G_i)_{k_i}{}^{\beta} (U_i)_{\beta}{}^{\alpha_i},\quad i = 1, \dots, d\\
		(C')_{\alpha_1 \dots \alpha_d} ={}& \sum_{\beta_1 = 1}^{t_1} \cdots \sum_{\beta_d = 1}^{t_d} C_{\beta_1 \dots \beta_d} (\inverse{U_1})^{\beta_1}{}_{\alpha_1} \cdots (\inverse{U_d})^{\beta_d}{}_{\alpha_d}.
	\end{align}
\end{proposition}

Many of the arguments in this section are the same as in \cref{sec:CP tensors}, so we do not repeat them here but just refer back.

\subsection{Smooth manifold}

Let $\inverse{\operatorname{mrank}}(t_1, \dots, t_d)$ denote the set of tensors with multilinear rank $(t_1, \dots, t_d)$.
Like $\inverse{\operatorname{ttrank}}(s_1, \dots, s_{d - 1})$ and $\inverse{\operatorname{rank}}(r)$, its closure is an algebraic variety.

\begin{lemma}
	Let $t_1 = t_2 \cdots t_d$.
	Then
	\begin{align}
		G \actson \Lambda_{t_1 \dots t_d} := \inverse{\operatorname{mrank}}(t_1, \dots, t_d)
	\end{align}
	is a transitive action.
\end{lemma}

\begin{proof}
	Let $E_i$ be the identity matrix $I_{n_i \times t_i}$ seen as an element of $\reals^{n_i} \otimes \reals^{t_i}$ and let $I$ be the identity matrix $I_{t_1 \times t_1}$ seen as an element of $\reals^{t_1} \otimes \dots \otimes \reals^{t_d}$.
	Then define
	\begin{align}
		T_{k_1 \dots k_d} = \sum_{\alpha_1, \dots, \alpha_d} I_{\alpha_1 \dots \alpha_d} (E_1)_{k_1}{}^{\alpha_1} \cdots (E_d)_{k_d}{}^{\alpha_d}.
	\end{align}
	Now, fix $C$ and $G_1$, \dots, $G_i$.
	By a previous comment, a Tucker decomposition with inner dimensions $t_1$, \dots, $t_d$ has $G_i$ with linearly independent columns.
	So let $\mathbar{G}_i \from \reals^{n_i \times (n_i - t_i)}$ be a basis completion to $G_i$.
	Moreover, the unfolding $C \from \reals^{t_1} \to \reals^{t_2} \otimes \dots \otimes \reals^{t_d}$ is an invertible $t_1 \times t_1$ matrix.
	The tensor
	\begin{align}
		S_{k_1 \dots k_d} = \sum_{\alpha_1, \dots, \alpha_d} C_{\alpha_1 \dots \alpha_d} (G_1)_{k_1}{}^{\alpha_1} \cdots (G_d)_{k_d}{}^{\alpha_d}.
	\end{align}
	is thus reached by the group element $g = (\begin{bmatrix} G_1 & \mathbar{G}_1 \end{bmatrix} \begin{bmatrix} C & \\ & 1 \end{bmatrix}, \begin{bmatrix} G_2 & \mathbar{G}_2 \end{bmatrix}, \dots, \begin{bmatrix} G_d & \mathbar{G}_d \end{bmatrix})$.
\end{proof}

There are other cases than $t_1 = t_2 \cdots t_d$ where $\inverse{\operatorname{mrank}}(t_1, \dots, t_d)$ has an open and dense orbit.
However, describing those orbits is more work and involves the so-called \emph{castling transform} of the factors, which generalizes the statement that $\reals^{p} \otimes \reals^{q} \otimes \reals^{r}$ has an open orbit iff $\reals^{p} \otimes \reals^{q} \otimes \reals^{p q - r}$ has an open orbit.
See Venturelli~\cite{Venturelli18} or Landsberg~\cite[Section 10.2.2]{Landsberg12} for details.
We also repeat the observation from the introduction: that there is typically more degrees of freedom in $\reals^{t_1} \otimes \dots \otimes \reals^{t_d}$, namely $t_1 \cdots t_d$ degrees, than in $\mathrm{GL}(t_1) \times \dots \times \mathrm{GL}(t_d)$, namely $t_1^{2} + \dots + t_d^{2}$ degrees.
This imposes the restriction
\begin{align}
	t_{1} \cdots t_{d} \leq t_{1}^{2} + \dots + t_{d}^{2}.
\end{align}
Furthermore, since $t_1 > t_2 \cdots t_d$ is not a possible multilinear rank, we also have the restriction
\begin{align}
	t_{1} \leq t_{2} \cdots t_{d}.
\end{align}
Solving for $t_{1}$, we find that it must satisfy
\begin{align}
	\frac{t_2 \cdots t_d + \sqrt{(t_2 \cdots t_d)^{2} - 4 (t_2^{2} + \dots + t_d^{2})}}{2} \leq{}& t_1 \leq t_2 \cdots t_d.
\end{align}
This domain is typically very small.
For $d = 3$, $t_2 = t_3 = 10$ for example, we have $98 \leq t_1 \leq 100$.
We therefore settle for describing the case $t_1 = t_2 \dots t_d$.

From \cref{prop:tucker_uniqueness} we can directly compute the stabilizer.

\begin{lemma}\label{prop:tucker_stabilizer}
	The stabilizer $H$ of $G \actson \Lambda_{t_1 \dots t_d}$ consists of elements of the form
	\begin{align}\label{eq:tucker_stabilizer}
		\begin{bmatrix}
			A_2^{-\mathsf{T}} \otimes \dots \otimes A_d^{-\mathsf{T}} & M_1\\
			& B_1
		\end{bmatrix}
		\times
		\begin{bmatrix}
			A_2 & M_2\\
			& B_2
		\end{bmatrix}
		\times \dots \times
		\begin{bmatrix}
			A_d & M_d\\
			& B_d
		\end{bmatrix}
	\end{align}
	where the $A_i$ are invertible $t_i \times t_i$ matrices, the $B_i$ are invertible $(n_i - t_i) \times (n_i - t_i)$ matrices, and $M_i$ are $t_i \times (n_i - t_i)$ matrices.
\end{lemma}

\begin{theorem}\label{thm:tucker_quotient}
	The set of tensors with multilinear rank $(t_{1}, \dots, t_{d})$, where $t_{1} = t_{2} \cdots t_{d}$, is a smooth homogeneous manifold,
	\begin{align}
		\Lambda_{t_1 \dots t_d} = G / H.
	\end{align}
\end{theorem}

\subsection{Representatives}

Similarly to \cref{sub:cp_representatives}, if $g = (g_1, \dots, g_d) \in G$, then we only need to store the first $t_i$ columns of $g_i$.

\subsection{Riemannian manifold}

Similarly to \cref{sub:cp_riemannian_manifold}, we can take the derivative of the expression \cref{eq:tucker_stabilizer} in \cref{prop:tucker_stabilizer} to get $H$'s Lie algebra, $\mathfrak{h}$.
It consists of elements $Y = (Y_1, \dots, Y_d)$ where the $Y_i$ are on block form
\begin{align}
	Y_1 ={}&
	\begin{bmatrix}
		-K_2^{\mathsf{T}} \otimes 1 \otimes \dots \otimes 1 - \dots - 1 \otimes \dots \otimes 1 \otimes K_d^{\mathsf{T}} & *\\
		& *
	\end{bmatrix},\nonumber\\
	Y_i ={}&
	\begin{bmatrix}
		K_i & *\\
		& *
	\end{bmatrix},\quad 2 \leq i \leq d,
\end{align}
for arbitrary $t_i \times t_i$ matrices $K_i$.

$\mathfrak{h}$'s orthogonal complement, $\mathfrak{m}$, consists of elements $X = (X_1, \dots, X_d)$ where the $X_i$ are on block form
\begin{align}
	X_1 ={}&
	\begin{bmatrix}
		L_1 & \\
		* & 0
	\end{bmatrix},\nonumber\\
	X_i ={}&
	\begin{bmatrix}
		L_i & \\
		* & 0
	\end{bmatrix},\quad 2 \leq i \leq d,
\end{align}
such that
\begin{align}\label{eq:tucker_invariant_subspace_candidate}
	L_i ={}& \operatorname{tr}_i L_1,\quad 2 \leq i \leq d,
\end{align}
where $\operatorname{tr}_{i}(A_1 \otimes \dots \otimes A_d) = (\operatorname{tr} A_1) \cdots \mathhat{\operatorname{tr} A_i} \cdots (\operatorname{tr} A_d) \transpose{A_i}$.
The easiest way to see this restriction on $L_i$ is to write the inner product as
\begin{align}\label{eq:tucker_inner_product}
	\innerproduct{Y}{X} ={}& \innerproduct{K_2}{L_1 - \operatorname{tr}_2 L_1} + \innerproduct{K_3}{L_2 - \operatorname{tr}_2 L_1} + \dots + \innerproduct{K_d}{L_{d} - \operatorname{tr}_{d} L_{1}}
\end{align}
and note that this should hold for all $K_i$ separately.
Note that $L_2$, \dots, $L_d$ are completely determined by $L_1$.

Like in \cref{sub:cp_riemannian_manifold}, we consider the right-invariant metric on $G$ induced by the Euclidean inner product on $\mathfrak{g}$, and define the metric on $\Lambda_{t_1 \dots t_d}$ by demanding that $\pi \from G \to \Lambda_{t_1 \dots t_d}$ is a Riemannian submersion.
We do not give a full description of the horizontal space at a general point, but just note that, for each $i$, the same argument that was used in \cref{sub:cp_riemannian_manifold} can be used to derive a rank $t_{i}$ decomposition similar to \cref{eq:horizontal_rank_decomposition}.

\subsection{Riemannian homogeneous manifold}

\begin{proposition}
	$\Lambda_{t_1 \dots t_d}$ is reductive if and only if $n_1 = t_1$, \dots, $n_d = t_d$.
	Moreover, when $\Lambda_{t_1 \dots t_d}$ is reductive, $\mathfrak{m}$ is an invariant subspace.
\end{proposition}

\begin{proof}
	Assume $n_1 = t_1$, \dots, $n_d = t_d$ and let $h = (\transpose{(A_2 \otimes \dots \otimes A_d)}, A_2, \dots, A_d) \in H$.
	We have to show that \cref{eq:tucker_invariant_subspace_candidate} is preserved by $L_i \mapsto L_i' = A_i L_i \inverse{A_i}$.
	This can be seen from
	\begin{align}
		L_i' ={}& A_i L_i \inverse{A_i}\nonumber\\
		={}& A_i (\operatorname{tr}_i L_1) \inverse{A_i}\nonumber\\
		={}& \operatorname{tr}_i [(A_1 \otimes \dots \otimes A_d)^{-\mathsf{T}} L_1 \transpose{(A_1 \otimes \dots \otimes A_d)}] \nonumber\\
		={}& \operatorname{tr}_i L_1'.
	\end{align}
	
	If $n_i \neq t_i$ for some $i$, then it is possible to show that $\Lambda_{t_1 \dots t_d}$ is not reductive with the same argument as in \cref{prop:cp_reductive}.
\end{proof}

\subsection{Geodesics}

By an argument completely analogous to \cref{prop:cp_efficient_geodesics}, we have the following result.

\begin{proposition}
	Given a tangent vector $X \in T_p \Lambda_{t_1 \dots t_d}$, $\operatorname{exp}_p(X)$ can be estimated using
	\begin{align}
		\sum_{i = 1}^{d} \left[ \frac{110}{3} n_i t_i^{2} + (146 + 36 z_i) t_i^{3} + \order{n_i t_i + t_i^{2}} \right] \textrm{ basic operations}
	\end{align}
	where $z_i = \lceil \operatorname{log}_{2}\norm{X_i \inverse{g_{i}}} \rceil + 2$.
\end{proposition}

\section{The tensor train manifold}%
\label{sec:Tensor trains}

\begin{proposition}\label{prop:tt_uniqueness}
	The TT decomposition is unique up to a change of basis in $\reals^{s_1}$, \dots, $\reals^{s_{d - 1}}$.
	More precisely,
	\begin{align}
		\sum_{\alpha_1}^{s_1} \cdots \sum_{\alpha_{d - 1}}^{s_{d - 1}} (F_1')_{k_1 \alpha_1} (F_2')^{\alpha_1}{}_{k_2 \alpha_2} \cdots (F_d')^{\alpha_{d - 1}}{}_{k_d} = \sum_{\alpha_1}^{s_1} \cdots \sum_{\alpha_{d - 1}}^{s_{d - 1}} (F_1)_{k_1 \alpha_1} (F_2)^{\alpha_1}{}_{k_2 \alpha_2} \cdots (F_d)^{\alpha_{d - 1}}{}_{k_d}
	\end{align}
	iff there are matrices $U_1 \in \mathrm{GL}(s_1)$, \dots, $U_{d - 1} \in \mathrm{GL}(s_{d - 1})$ such that
	\begin{align}
		(F_1')_{k_i \alpha_i} ={}& \sum_{\beta = 1}^{s_1}  (F_1)_{k_1 \beta} (\inverse{U_{2}})^{\beta}{}_{\alpha_1},\\
		(F_i')^{\alpha_{i - 1}}{}_{k_i \alpha_i} ={}& \sum_{\beta = 1}^{s_{i - 1}} \sum_{\gamma = 1}^{s_i} (U_{i - 1})_{\beta}{}^{\alpha_{i - 1}} (F_i)^{\beta}{}_{k_i \gamma} (\inverse{U_{i}})^{\gamma}{}_{\alpha_i}, \quad 2 \leq i \leq d - 1,\\
		(F_d')^{\alpha_{d - 1}}{}_{k_d} ={}& \sum_{\beta = 1}^{s_{d - 1}} (U_{d - 1})_{\beta}{}^{\alpha_{d - 1}} (F_d)^{\beta}{}_{k_d}.
	\end{align}
\end{proposition}

Uschmajew and Vandereycken~\cite[proposition 3]{Uschmajew13} shows that this holds for Hierarchical Tucker decompositions, of which the TT decomposition is a special case.

\subsection{Smooth manifold}

Let $\inverse{\operatorname{ttrank}}(s_{1}, \dots, s_{d - 1})$ denote the set of tensors with TT rank $(s_{1}, \dots, s_{d - 1})$.
Like $\inverse{\operatorname{rank}}(r)$, its closure is an algebraic variety.

\begin{lemma}\label{prop:tt_open_orbit}
	Let $s_{1} \leq n_1$, $s_{i - 1} s_{i} \leq n_i$ for all $2 \leq i \leq d - 1$, and $s_{d - 1} \leq n_{d}$.
	Then
	\begin{align}
		G \actson \inverse{\operatorname{ttrank}}(s_1, \dots, s_{d - 1})
	\end{align}
	has an open dense orbit, $\Pi_{s_1 \dots s_{d - 1}}$, consisting of elements with maximal multilinear rank.
\end{lemma}

\begin{remark}
	Maximal multilinear rank in $\inverse{\operatorname{ttrank}}(s_1, \dots, s_{d - 1})$ is $(s_1, s_1 s_2, \dots, s_{d - 1})$.
\end{remark}

\begin{proof}
	From \cref{eq:tt_rank}, it is clear that TT rank is preserved under $G$.
	The action is thus well-defined.
	
	Let $E_{i} \from \reals^{s_{i - 1}} \otimes \reals^{n_i} \otimes \reals^{s_{i}}$ be the tensor whose unfolding is the identity matrix, $(E_{i} \from \reals^{n_i} \to \reals^{s_{i - 1}} \otimes \reals^{s_{i}}) = I_{n_i \times s_{i - 1} s_i}$, and define
	\begin{align}
		T_{k_1 k_2 \dots k_d} = \sum_{\alpha_1, \dots, \alpha_{d - 1}} (E_1)_{k_1 \alpha_1} (E_2)^{\alpha_1}{}_{k_2 \alpha_2} \cdots (E_d)^{\alpha_{d - 1}}{}_{k_d}.
	\end{align}
	Similarly to the proof of \cref{prop:cp_tensor_open_orbit}, we now want to show that $\Pi_{s_{1} \dots s_{d - 1}} = G \cdot T$ is Zariski open in $\mathbar{\inverse{\operatorname{ttrank}}(s_{1}, \dots, s_{d - 1})}$.

	Any other tensor
	\begin{align}
		S_{k_1 \dots k_d} = \sum_{\alpha_1, \dots, \alpha_{d - 1}} (F_1)_{k_1 \alpha_1} (F_2)^{\alpha_1}{}_{k_2 \alpha_2} \cdots (F_d)^{\alpha_{d - 1}}{}_{k_d}
	\end{align}
	where the matrices $F_i \from \reals^{n_i \times s_{i - 1} s_i}$ have full rank is reached by the group element $(g_1, \dots, g_d)$ where $g_i = \begin{bmatrix} F_i & \mathbar{F}_i \end{bmatrix}$ such that $\mathbar{F}_i \from \reals^{n_i \times (n_i - s_{i - 1} s_{i})}$ is a basis completion.
	Moreover, the action of $G$ preserves the rank of $F_i$.
	The condition that the $F_i$ have full rank is a Zariski open condition.

	To show that elements outside of $\Pi_{s_{1} \dots s_{d - 1}}$ don not have maximal multilinear rank, note that the multilinear rank of $S$ is just $(\operatorname{rank} F_1, \dots, \operatorname{rank} F_d)$.
\end{proof}

From \Cref{prop:tt_uniqueness} we can directly compute the stabilizer.

\begin{lemma}\label{prop:tt_stabilizer}
	The stabilizer $H$ of $G \actson \Pi_{s_1 \dots s_{d - 1}}$ consists of elements of the form
	\begin{align}\label{eq:tt_stabilizer}
		\begin{bmatrix}
			A_1 & M_1\\
			& B_1
		\end{bmatrix}
		\times
		\begin{bmatrix}
			A_1^{-\mathsf{T}} \otimes A_2 & M_2\\
			& B_2
		\end{bmatrix}
		\times \dots \times
		\begin{bmatrix}
			A_{d - 1}^{-\mathsf{T}} & M_d\\
			& B_d
		\end{bmatrix}
	\end{align}
	where the $A_i$ are invertible $s_i \times s_i$ matrices, the $B_i$ are invertible $(n_i - s_{i - 1}s_{i}) \times (n_i - s_{i - 1} s_{i})$ matrices, and the $M_i$ are $s_{i - 1} s_{i} \times (n_i - s_{i - 1} s_{i})$ matrices.
\end{lemma}

Combining \cref{prop:tt_open_orbit,prop:tt_stabilizer}, we have the main result of this subsection.

\begin{theorem}\label{thm:tt_quotient}
	The set of tensors with TT rank $(s_{1}, \dots, s_{d - 1})$ and multilinear rank $(s_{1},$ $s_{1} s_{2}, \dots, s_{d - 1})$ is a smooth homogeneous manifold,
	\begin{align}
		\Pi_{s_1 \dots s_{d - 1}} = G / H.
	\end{align}
\end{theorem}

\subsection{Representatives}%
\label{sub:tt_representatives}

Similarly to \cref{sub:cp_representatives}, if $g = (g_1, \dots, g_d) \in G$, then we only need to store the first $s_{i - 1} s_{i}$ columns of $g_i$.

\subsection{Riemannian manifold}%
\label{sub:tt_riemannian_manifold}

Similarly to \cref{sub:cp_riemannian_manifold}, we can take the derivative of the expression \cref{eq:tt_stabilizer} in \cref{prop:tt_stabilizer} to get $H$'s Lie algebra, $\mathfrak{h}$.
It consists of elements $Y = (Y_1, \dots, Y_d)$ where the $Y_i$ are on block form
\begin{align}
	Y_1 ={}& \begin{bmatrix} K_1 & *\\ & * \end{bmatrix},\nonumber\\
	Y_i ={}& \begin{bmatrix} -\transpose{K_{i - 1}} \otimes 1 + 1 \otimes K_{i} & *\\ & * \end{bmatrix}, \quad 2 \leq i \leq d - 1,\nonumber\\
	Y_d ={}& \begin{bmatrix} -\transpose{K_{d - 1}} & *\\ & * \end{bmatrix},
\end{align}
for arbitrary $s_{i} \times s_{i}$ matrices  $K_i$.
The orthogonal complement, $\mathfrak{m}$, consists of elements $X = (X_1, \dots, X_d)$ where the $X_i$ are on block form
\begin{align}
	X_1 ={}& \begin{bmatrix} L_1 & \\ * & 0 \end{bmatrix},\nonumber\\
	X_i ={}& \begin{bmatrix} L_i & \\ * & 0 \end{bmatrix}, \quad 2 \leq i \leq d - 1,\nonumber\\
	X_d ={}& \begin{bmatrix} L_d & \\ * & 0 \end{bmatrix},
\end{align}
such that
\begin{align}
	L_1 ={}& \operatorname{tr}_2 L_2\nonumber\\
	\operatorname{tr}_1 L_{i - 1} ={}& \operatorname{tr}_2 L_i, \quad 2 \leq i \leq d - 1\nonumber\\
	\operatorname{tr}_1 L_{d - 1} ={}& L_d,%
	\label{eq:tt_invariant_subspace_candidate}
\end{align}
where $\operatorname{tr}_1 (A \otimes B) = (\operatorname{tr} A) B$ and $\operatorname{tr}_2 (A \otimes B) = (\operatorname{tr} B) \transpose{A}$.
This can be shown the in the same way as \cref{eq:tucker_invariant_subspace_candidate}.

Like in \cref{sub:cp_riemannian_manifold}, we consider the right-invariant metric on $G$ induced by the Euclidean inner product on $\mathfrak{g}$, and define the metric on $\Pi_{s_{1} \dots s_{d - 1}}$ by demanding that $\pi \from G \to \Pi_{s_{1} \dots s_{d - 1}}$ is a Riemannian submersion.
We do not give a full description of the horizontal space at a general point $g$, but just note that the same argument that was used in \cref{sub:cp_riemannian_manifold} can be used to derive a rank $s_{i - 1} s_{i}$ decomposition similar to \cref{eq:horizontal_rank_decomposition},

\subsection{Riemannian homogeneous manifold}

\begin{proposition}\label{prop:tt_reductive}
	$\Pi_{s_1 \dots s_{d - 1}}$ is reductive if and only if $n_1 = s_1$, $n_2 = s_1 s_2$, \dots, $n_d = s_{d - 1}$.
	Moreover, when $\Pi_{s_1 \dots s_{d - 1}}$ is reductive, $\mathfrak{m}$ is an invariant subspace.
\end{proposition}

\begin{proof}
	Assume $n_1 = s_1$, $n_2 = s_1 s_2$, \dots, $n_d = s_{d - 1}$ and let $h = (A_1, A_1^{-\mathsf{T}} \otimes A_2, \dots, A_{d - 1}^{-\mathsf{T}}) \in H$.
	We have that $X \mapsto h X \inverse{h}$ maps $L_i \mapsto L_i' = (A_{i - 1}^{-\mathsf{T}} \otimes A_{i}) L_i (\transpose{A_{i - 1}} \otimes \inverse{A_{i}})$.
	To show that $\mathfrak{m}$ is invariant, we need to argue that the relation \cref{eq:tt_invariant_subspace_candidate} is preserved.
	This can be seen from
	\begin{align}
		\operatorname{tr}_{1} L_{i - 1}' ={}& \operatorname{tr}_1 [(A_{i - 2}^{-\mathsf{T}} \otimes A_{i - 1}) L_{i - 1} (\transpose{A_{i - 2}} \otimes \inverse{A_{i - 1}})]\nonumber\\
		={}& A_{i - 1} (\operatorname{tr}_1 L_{i - 1}) \inverse{A_{i - 1}}\nonumber\\
		={}& A_{i - 1} (\operatorname{tr}_2 L_{i}) \inverse{A_{i - 1}}\nonumber\\
		={}& \operatorname{tr}_2 [(A_{i - 1}^{-\mathsf{T}} \otimes A_{i}) L_{i} (\transpose{A_{i - 1}} \otimes \inverse{A_{i}})]\nonumber\\
		={}& \operatorname{tr}_2 L_i'.
	\end{align}

	On the other hand, if $n_i \neq s_{i - 1} s_{i}$ for some $i$, then it is possible to show that $\Pi_{s_{1} \dots s_{d - 1}}$ is not reductive with the same argument as in \cref{prop:cp_reductive}.
\end{proof}

\subsection{Geodesics}%

By an argument completely analogous to \cref{prop:cp_efficient_geodesics}, we have the following result.

\begin{proposition}\label{prop:tt_efficient_geodesics}
	Given a tangent vector $X \in T_p \Pi_{s_1 \dots s_{d - 1}}$, $\operatorname{exp}_{p}(X)$ can be estimated using
	\begin{align}
		&\frac{110}{3} n_1 s_1^{2} + (146 + 36 z_1) s_1^{3} + \order{n_1 s_1 + s_1^{2}}\nonumber\\
		{}+{}& \sum_{i = 2}^{d - 1} \frac{110}{3} n_i (s_{i - 1} s_i)^{2} + (146 + 36 z_i) (s_{i - 1} s_i)^{3} + \order{n_i s_{i - 1} s_i + (s_{i - 1} s_i)^{2}}\nonumber\\
		{}+{}& \frac{110}{3} n_d s_{d - 1}^{2} + (146 + 36 z_d) s_{d - 1}^{3} + \order{n_d s_{d - 1} + s_{d - 1}^{2}} \textrm{ basic operations}
	\end{align}
	where $z_i = \lceil \operatorname{log}_2 \norm{X_i \inverse{g_{i}}} \rceil + 2$.
\end{proposition}

\appendix

\section{Error bound for Padé approximant}%
\label{sec:Error bound for Pade approximant}

Recall the definition of $\psi_{1}$ from \cref{sub:cp_geodesics} and the definition of $r_{66}$ from \cref{eq:pade_approximant}, and let $\norm{\cdot}$ be a matrix norm.

\begin{lemma}\label{lemma:pade_approximant}
	If $\norm{M} \leq 1 / 2$, then $r_{66}(M)$ approximates $\psi_{1}(M)$ to within double precision $2^{-53} \approx \num{e-16}$.
\end{lemma}

\begin{proof}
	$r_{66}(M)$ has a Taylor expansion $\sum_{i = 1}^{\infty} \frac{r_{66}^{(n)}(0)}{n!} M^{n}$.
	Since $r_{66}$ is the Padé approximant to $\psi_{1}(M) = \sum_{i = 1}^{\infty} \frac{1}{(n + 1)!} M^{n}$, they agree up to the first $12$ terms.
	We thus want to bound the series
	\begin{align}
		\psi_{1}(M) - r_{66}(M) = \sum_{n = 13}^{\infty} \frac{1 - (n + 1) r_{66}^{(n)}(0)}{(n + 1)!} M^{n}.
	\end{align}
	
	First, we compute term $13$ and $14$ manually.
	They are $M^{13} / 149597947699200$ and $181 M^{14} / 29171599801344000$ respectively.

	Second, we note that the $n$th derivative of $r_{66}$ is a rational function $p_{n}(M) / q_{n}(M)$.
	The quotient rule gives the recurrence relation $p_{n + 1} = p_{n}' q_{n} - p_{n} q_{n}'$ and $q_{n + 1} = q_{n}^{2}$.
	We have that $q_{n}(0) = 1$ for all $n$.
	Moreover, $\norm{p_{0}^{(k)}(0)}$, $\norm{q_{0}^{(k)}(0)} \leq 1 / 2^{k}$ for all $k$.
	Using the triangle inequality in the recurrence relation, this implies that $\norm{r_{66}^{(n)}(0)} = \norm{p_{n}(0)} \leq 1$.
	Thus
	{\hfuzz=38pt%
	\begin{align}
		\norm{\psi_{1}(M) - r_{66}(M)} ={}& \norm{\sum_{n = 13}^{\infty} \frac{1 - (n + 1) r_{66}^{(n)}(0)}{(n + 1)!} M^{n}}\nonumber\\
		\leq{}& \frac{\norm{M}^{13}}{149597947699200} + \frac{181 \norm{M}^{14}}{29171599801344000} + \sum_{n = 15}^{\infty} \frac{1 + (n + 1) \norm{r_{66}^{(n)}(0)}}{(n + 1)!} \norm{M}^{n}\nonumber\\
		\leq{}& \frac{(1 / 2)^{13}}{149597947699200} + \frac{181 (1 / 2)^{14}}{29171599801344000} + \sum_{n = 15}^{\infty} \frac{1 + (n + 1)}{(n + 1)!} (1 / 2)^{n}\nonumber\\
		={}& 3 \sqrt{e} - \frac{2364006584786656554317}{477947491145220096000}\nonumber\\
		\leq{}& \num{2.7e-17}.
	\end{align}
	}
\end{proof}

\begin{lemma}\label{lemma:pade_approximant_mexp}
	If $\norm{M} \leq 1 / 2$, then $\operatorname{mexp}(M)$ is approximated by its Padé approximant to within double precision.
\end{lemma}

We do not prove \cref{lemma:pade_approximant_mexp} here, but instead refer to Higham~\cite[Section 10.3]{Higham08}.

\section{Appendix: Counting the operations}%
\label{sec:Appendix: Counting the operations}

We now prove \cref{prop:cp_efficient_geodesics}.

\begin{proof}~
	\begin{itemize}
		\item $20 n_{i} r^{2} / 3 + 22 r^{3} / 3$ operations to build $\Gamma_{12}$ using \cref{eq:Gamma12}:
		\begin{itemize}
			\item $8 n_{i} r^{2} / 3$ operations to divide $g_{21}$ by $g_{11}$,
			\item $2 n_{i} r^{2}$ operations to multiply $g_{11}^{-\mathsf{T}} \transpose{g_{21}}$ with $g_{21} \inverse{g_{11}}$,
			\item $2 r^{3}$ operations to square $g_{11}^{-\mathsf{T}} \transpose{g_{21}} g_{21} \inverse{g_{11}}$,
			\item $8 r^{3} / 3$ operations to divide by $1 + g_{11}^{-\mathsf{T}} \transpose{g_{21}} g_{21} \inverse{g_{11}}$,
			\item $8 r^{3} / 3$ operations to divide outside the parenthesis from the left by $g_{11}$,
			\item $2 n_{i} r^{2}$ to operations to multiply from the right by $g_{11}^{-\mathsf{T}} \transpose{g_{21}}$,
		\end{itemize}
		\item no extra operations to build $A$,
		\item $2 n_{i} r^{2} + 8 r^{3} / 3$ operations to build $B$:
		\begin{itemize}
			\item $2 n_{i} r^{2}$ operations to multiply $\Gamma_{12}$ with $g_{21}$,
			\item $8 r^{3} / 3$ operations to divide $1 - \Gamma_{12} g_{21}$ by $g_{11}$ from the right,
		\end{itemize}
		\item $2 n_{i} r^{2}$ operations to multiply $B$ and $A$,
		\item $(52 / 3 + 4 (z_{i} - 1)) r^{3}$ operations to evaluate $\psi_{1}$ on $B A$ using \cref{eq:pade_approximant,eq:scaling_and_squaring},
		\begin{itemize}
			\item $12 r^{3}$ operations to form the powers $2^{-z} M$, \dots, $(2^{-z} M)^{6}$,
			\item $8 r^{3} / 3$ operations to evaluate the quotient,
			\item $8 r^{3} / 3$ operations to evaluate $\operatorname{mexp}(2^{-z} M)$,
			\item $2 (z - 1) r^{3}$ operations to form $\operatorname{mexp}(2^{-z + 1} M)$, \dots, $\operatorname{mexp}(2^{-1} M)$ by recursively using $\operatorname{mexp}(2 M) = \operatorname{mexp}(M) \operatorname{mexp}(M)$,
			\item $2 (z - 1) r^{3}$ operations to multiply the factors in \cref{eq:scaling_and_squaring},
		\end{itemize}
		\item no extra operations to build $A'$ or $B'$,
		\item $2 n_{i} (2 r)^{2}$ operations to multiply $B'$ and $A'$,
		\item $(52 / 3 + 4 (z_{i} - 1)) (2 r)^{3}$ operations to evaluate $\psi_{1}$ on $B' A'$,
		\item no extra operations to build the first term in \cref{eq:geodesics_efficient},
		\item $4 n_{i} r^{2} + 2 r^{3}$ operations to build the second term in \cref{eq:geodesics_efficient}:
		\begin{itemize}
			\item $2 n_{i} r^{2}$ operations to multiply $B$ with $\begin{bmatrix} g_{11}\\ g_{21} \end{bmatrix}$,
			\item $2 r^{3}$ operations to multiply $\psi_{1}(B A)$ with $B \begin{bmatrix} g_{11}\\ g_{21} \end{bmatrix}$,
			\item $2 n_{i} r^{2}$ operations to multiply with $A$ from the left,
		\end{itemize}
		\item $8 n_{i} r^{2} + 8 r^{3}$ operations to build the third term in \cref{eq:geodesics_efficient}:
		\begin{itemize}
			\item $2 n_{i} (2 r) r$ operations to multiply $B'$ with $\begin{bmatrix} g_{11}\\ g_{21} \end{bmatrix}$,
			\item $2 (2 r)^{2} r$ operations to multiply $\psi_{1}(B' A')$ with $B' \begin{bmatrix} g_{11}\\ g_{21} \end{bmatrix}$,
			\item $2 n_{i} (2 r) r$ operations to multiply with $A'$ from the left,
		\end{itemize}
		\item $6 n_{i} r^{2} + 6 r^{3}$ operations to build the last term in \cref{eq:geodesics_efficient}:
		\begin{itemize}
			\item $2 n_i r (2 r)$ operations to multiply $B$ with $A'$,
			\item $2 r (2 r) r$ operations to multiply $B A'$ with $\psi_{1}(B' A') B' \begin{bmatrix} g_{11}\\ g_{21} \end{bmatrix}$,
			\item $2 r^{3}$ operations to multiply with $\psi_{1}(B A)$ from the left,
			\item $2 n_{i} r^{2}$ operations to multiply with $A$ from the left.
		\end{itemize}
	\end{itemize}
\end{proof}

{\hfuzz=22pt
\printbibliography
}
\end{document}